\documentclass[11pt]{amsart}
\usepackage[utf8]{inputenc}
\usepackage[T1]{fontenc}
\usepackage{amsmath,amsthm}
\usepackage{amsfonts,amssymb,enumerate}
\usepackage{url,paralist}
\usepackage{mathtools}  
\usepackage[colorlinks=true,urlcolor=blue,linkcolor=red,citecolor=magenta]{hyperref}
\usepackage{enumerate}

\usepackage{graphicx}
\graphicspath{ {./images/} }

\theoremstyle{plain}
\newtheorem{theorem}{Theorem}[section]
\newtheorem{lemma}[theorem]{Lemma}

\newtheorem{proposition}[theorem]{Proposition}

\newtheorem*{theorem*}{Theorem}

\newtheorem*{claim*}{Claim}

\theoremstyle{definition}

\usepackage{xcolor}

\numberwithin{equation}{section}

\begin{document}

\title[Singularity formation]{Singularity formation along the line bundle mean curvature flow}
\author[Y. H. Chan]{Yu Hin Chan}
 \email{yuhchan@ucdavis.edu}
\author[A. Jacob]{Adam Jacob*}
\thanks{$^{*}$Supported in part by a Simons Collaboration Grant.}

 \email{ajacob@math.ucdavis.edu}

 \address{Department of Mathematics, University of California Davis, 1 Shields Ave., Davis, CA, 95616}



\maketitle


\begin{abstract}  The line bundle mean curvature flow is a complex analogue of the mean curvature flow for Lagrangian graphs, with fixed points solving the deformed Hermitian-Yang-Mills equation.  In this paper we construct two distinct examples of singularities along the flow. First, we find a finite time singularity, ruling out long time existence of the flow in general. Next we show long time existence of the flow with a Calabi symmetry assumption on the blowup of $\mathbb P^n$, $n\geq 3$, if one assumes supercritical phase. Using this, we find an example where a singularity occurs at infinite time along the destabilizing subvariety in the semi-stable case.

 \end{abstract}

\section{Introduction}
Let $(X,\omega)$ be a compact K\"ahler manifold, and $[\alpha]\in H^{1,1}(X,\mathbb R)$ a real cohomology class.  The {\it deformed Hermitian-Yang-Mills} (dHYM) equation  seeks a representative $\alpha\in[\alpha]$ satisfying
\begin{equation}
\label{dHYM1}
{\rm Im}(e^{- i\hat\theta}(\omega+i\alpha)^n)=0
\end{equation}
for a fixed constant $e^{i\hat\theta}\in S^1$. Recently this equation has garnered significant attention, and extensive work has   centered around the relationship between existence of a solution and notions of geometric stability \cite{C1, CJY, CLSY, CS, CLT, JS, Ping}. Although much of this work has been done with elliptic methods, substantial progress has been made following a parabolic approach as well \cite{FYZ,HJ, JY, Tak2, Tak1}.

In this paper we focus on one such parabolic method, known as the line bundle mean curvature flow.  Fix a background form $\alpha_0\in[\alpha]$, and define $\alpha_t:=\alpha_0+i\partial\bar\partial\phi_t$.   At any point $p\in X$ one can choose coordinates so  $\omega^{-1}\alpha_t$ is diagonal with eigenvalues $\{\lambda_1,...,\lambda_n\}$. The line bundle mean curvature flow can   be expressed as
\begin{equation}
\label{lbmcf}
\dot\phi_t=\sum_k{\rm arctan}(\lambda_k)-\hat\theta,
\end{equation}
where $\hat\theta$ is some choice of a lift of $e^{i\hat\theta}$ to $\mathbb R$. This parabolic flow is the complex analogue of the Lagrangian mean curvature flow in the graphical setting, with the distinction being that the mean curvature flow is given by eigenvalues of the real Hessian of a function, as opposed to the complex Hessian (we direct the reader to \cite{S1,S2,SmW2, TTW, W} for further background on the Lagrangian   case). By the complex formulation of arctan, one sees $\sum_k{\rm arctan}(\lambda_k)$ is   the argument of the top dimensional form $(\omega+i\alpha)^n$, and so solutions to \eqref{dHYM1} are fixed points of \eqref{lbmcf}. We denote this argument  by $\Theta(\alpha_{t})=\sum_k{\rm arctan}(\lambda_k)$.

Developed by the second author and S.-T. Yau, the flow  \eqref{lbmcf} was used to prove existence of a solution to \eqref{dHYM1} under the assumption of hypercritical phase, defined by $\Theta(\alpha_0)>(n-1)\frac{\pi}2$, in addition to the assumption that $(X,\omega)$ has non-negative   orthogonal bisectional curvature \cite{JY}. The phase assumption is useful for two reasons. First, it ensures  convexity of the operator $\Theta(\cdot)$. Second, it allows a natural choice of a lift of $\hat\theta$, which is a priori defined up to a multiple of $2\pi$. In fact, being able to choose such a lift is a major difficulty in the study of \eqref{dHYM1}, and  one would not expect the flow to converge without making the appropriate choice of a lift at the start. 

Given the cohomological obstructions to the existence of solutions to \eqref{dHYM1} from \cite{CJY}, it is evident that the flow   \eqref{lbmcf} can not converge if one choses an initial, unstable class. However, it was previously not known if  the flow exists for all time, or if a finite time singularity could occur. The first goal of our paper is to   construct an explicit example of a finite time singularity, ruling out long time existence. 
\begin{theorem}\label{main}
	Let $X$ be the blowup of $\mathbb P^n$ at a point. There exists a K\"ahler form $\omega$, and cohomology  class $[\alpha]\in H^{1,1}(X,\mathbb R)$  admitting a representative $\alpha_0$, for which the flow   \ref{lbmcf} achieves a finite-time singularity. Specifically, if $\lambda_{Max}(p,t)$ denotes the largest eigenvalue of $\omega^{-1}\alpha_t$ at a point $p\in X$, then there exists a sequence of points $\{x_k\}\subset X$ and times $t_k\rightarrow T<\infty$ such that 
	$$\lim_{k\rightarrow\infty}\lambda_{Max}(x_k, t_k)=\infty.$$
\end{theorem}
This example is constructed using a particular type of symmetry on $X$, called Calabi Symmetry, which is described in Section \ref{symmetry}. The symmetry allows the dHYM equation to be written as an ODE, and the flow \eqref{lbmcf} is reduced to a parabolic PDE with one spacial variable. Due to similarities with the curve shortening flow, we construct subsolutions which, along with a particular choice of an initial condition,    force a singularity to happen.  Our example can be constructed on classes that admit a solution to the dHYM equation, demonstrating that finite time singularities can not be ruled out by class conditions alone. In fact, we believe similar examples of finite time singularities can be constructed on any pair of classes $[\omega]$ and $[\alpha]$ on the blowup of $\mathbb P^n$. Thus finite time singularities will remain an integral part of the study of \ref{lbmcf}, and will need to be ruled out by choosing appropriate initial conditions.

Our next goal is to demonstrate that the flow can also become singular at infinite time, and to find an example where we can predict exactly where this singularity will occur from the initial classes $[\alpha]$ and $[\omega]$. Using the same Calabi Symmetry setup as above, we first show that if the initial form satisfies supercritical phase then the flow exists for all time:
\begin{theorem}\label{main2}
	Let $(X,\omega)$ be the blowup of $\mathbb P^n$ at a point, $n\geq 3$, and consider a class $[\alpha]\in H^{1,1}(X,\mathbb R)$. Assume $\omega$, $\alpha_0\in[\alpha]$ have Calabi-symmetry, and furthermore assume  $\alpha_0$ has supercritical phase, that is  $\Theta(\alpha_0)>(n-2)\frac\pi 2$. Then the flow \eqref{lbmcf} beginning at $\alpha_0$ exists for all time.
 \end{theorem}
Note that Takahashi proved long time existence for the line bundle mean curvature flow in the hypercritical phase case \cite{Tak2}, which implies that $\alpha_t$ stays a K\"ahler form and the operator  $\Theta(\cdot)$ is convex. For our result we use the weaker supercritical phase assumption, which does  not imply convexity of the operator is $\Theta(\cdot)$. However, it does imply the level sets are convex \cite{CJY}, which is enough to apply Evans-Krylov. 

To see where the long time singularity occurs, we turn to the conjectured relationship between solutions to the dHYM equation and stability.  Following the work of  Lejmi-Sz\'ekelyhidi on the $J$-equation \cite{LS}, the second author, along with T.C. Collins and   S.-T. Yau, integrated a certain positivity condition along subvarieties to develop a necessary class condition for existence, and conjectured it was a sufficient condition as well  \cite{CJY}. Specifically, for any irreducible analytic subvariety $V\subseteq X$, define the complex number
\begin{equation}
Z_{[\alpha][\omega]}(V):=-\int_V e^{-i\omega+\alpha},\nonumber
\end{equation}
where by convention we only integrate the term in the expansion of order ${\rm dim}(V)$. Under the supercritical phase assumption $Z_{[\alpha][\omega]}(X)$ lies in the upper half plane $\mathbb H$.  The conjecture of Collins-J.-Yau posits that a solution to the dHYM equation exists if and only if 
\begin{equation}
\label{stability1}
 \pi>{\rm arg} Z_{[\alpha][\omega]}(V)>{\rm arg} Z_{[\alpha][\omega]}(X).
 \end{equation}

Later, when $n=3$, Collins-Xie-Yau demonstrated a necessary Chern number inequality \cite{CXY} (which has since been extended to $n=4$ \cite{HJ2}), which is also useful for defining the lifted angle $\hat\theta$ algebraically. Collins-Yau further conjectured that such a Chern number inequality in higher dimension was needed \cite{CY}.   Indeed, recently when $n=3$ an example was found where the stability inequality \eqref{stability1} holds, but the Chern number inequality does not, and no solution to the dHYM equation exists \cite{Z}. We note that slightly weaker versions of the Collins-J.-Yau conjecture have been solved by Chen \cite{C1} (assuming uniform stability), and Chu-Lee-Takahashi \cite{CLT} (in the projective case). These results all rest on the supercritical phase assumption.

Some of the few results without the supercritical phase assumption are due to the second author and Sheu \cite{JS}  (and later \cite{J}), who take advantage of the same Calabi-Symmetry used in this paper. They demonstrate that  the inequalities  \eqref{stability1} can be reinterpreted as stating whether two points in $\mathbb C$ lie on the same level set of a harmonic polynomial, from which it follows that solutions to the dHYM exists. Since the stability conditions are necessary, and supercritical phase  impiles long time existence of the flow in this setting, the unstable case ends up being the perfect setup to construct a singularity for $\alpha_t$ as $t$ approaches infinity. In particular we demonstrate:

\begin{theorem}\label{main3}
	Let $(X,\omega)$ be the blowup of $\mathbb P^n$ at a point, $n\geq 3$. There exists classes $[\alpha]$ and $[\omega]$, which are semi-stable in the sense of \eqref{stability1}, where   the flow \eqref{lbmcf} starting at an initial representative $\alpha_0$ exists for all time and  becomes singular at time $t=\infty$ along the destabilizing subvariety.\end{theorem}

Similar to the proof of Theorem \ref{main},  we utilize explicit subsolutions of a modified curve shortening flow to force a singularity at infinite time.

Here we briefly discuss two other parabolic flows in the literature for which solutions to  \eqref{dHYM1}  are   fixed points. The first is the tangent Lagrangian phase flow, introduced by Takahashi in \cite{Tak1}. Defined by $\dot\phi_t={\rm tan}(\Theta(\alpha_t)-\hat\theta),$ this flow is the the gradient flow of the Kempf-Ness functional arising from the infinite dimensional GIT picture for stability and the dHYM equation, as  developed by Collins-Yau \cite{CY}. As a result, this flow is well behaved with respect to many important functionals, and it could  be useful when exploring if   some type of limiting Harder-Narasimhan filtration exists in the unstable case. One downside of this flow is that it is only defined for ``almost calibrated'' potentials, when the angle $\Theta(\alpha_t)$ varies from the target angle by less than $\frac\pi 2$.  The second flow  was introduced by Fu-Yau-Zhang and is defined by  the equation $\dot\phi_t={\rm cot}(n\frac\pi 2-\Theta(\alpha_t))-{\rm cot}(n\frac\pi 2-\hat\theta)$ \cite{FYZ}. This flow has the advantage that ${\rm cot}(n\frac\pi 2-\Theta(\alpha_t))$ is concave under the supercritical phase assumption. Additionally the form of the flow allows for some useful estimates for subsolutions. However, this flow is only defined for supercritical phase, since otherwise one may end up taking the cotangent of zero. Note that in comparison to the above flows, the line bundle mean curvature flow is always defined for short time.

The singularity examples we construct point towards many new interesting problems to explore. One question is whether similar singularities can be constructed on more general K\"ahler manifolds, perhaps with some sort of gluing technique. We also wonder if there are any examples which can be related to singularities of the graphical Lagrangian mean curvature flow, given the formal similarities between the two flows. In the above examples, the highest eigenvalue of $\omega^{-1}\alpha_t$ approaches infinity while the derivatives of the eigenvalues stay bounded. Thus  the analogue of the graphical ``Lagrangian'' (given by one derivative of the potential) is tilting up to achieve vertical slope. It would be interesting if one could find examples with higher order singularities, which would allow for  a richer blowup analysis. Finally, in our long time singularity example, the singularity occurs along the destabilizing subvariety, and one would expect this relationship between stability and singularity formation to hold in more general settings. We hope to explore  these problems in future work.

The paper is organized as follows. In Section \ref{symmetry} we   introduce the Calabi symmetry assumption, which is used in constructing our examples. In Section \ref{finite} we construct our finite time singularity. Section \ref{long} contains our proof of long time existence in the supercritical phase case. We conclude with our example of a long time singularity in Section \ref{longsing}.

\medskip

{\bf Acknowledgements.} This problem arose at the American Institute of Mathematics workshop ``Stability in mirror symmetry,'' and the second author would like to thank the organizers and  the institute for providing a productive research environment. In particular special thanks   to Tristan C. Collins, Jason Lotay, and   Felix Schulze for some helpful discussion. This work was funded in part by a Simons collaboration grant. 

\section{Calabi Symmetry}
\label{symmetry}
Throughout this paper we work on  the blowup of $\mathbb P^n$ at one point $p$, which we denote by $X$.  This manifold admits $(1,1)$ forms that satisfy a symmetry ansatz, originally defined by Calabi in \cite{Ca}. We include a short introduction to this ansatz here, and direct the reader to \cite{Ca, FL, JS, Song, SW2, SW3, SW4, SY} for further details.

 Let $E$ denote the  exceptional divisor, and $H$  the pullback of the hyperplane divisor from  $\mathbb P^n$. These two divisors span $H^{1,1}(X,\mathbb R)$, and any K\"ahler class will lie in $a_1[H]-a_2[E]$ with $a_1>a_2>0$. Normalizing,  assume $X$ admits a K\"ahler form $\omega$ in the class $$[\omega]= a[H]-[E],$$
with $a>1$. Furthermore, assume our class $[\alpha]$ satisfies $$[\alpha]= p[H]-q[E],$$ for a choice of $p,q\in\mathbb R$. 

  On $X\backslash (H\cup E)\cong \mathbb C^n\backslash \{0\}$, set $\rho={\rm log}(|z|^2)$. If $u(\rho)\in C^\infty(\mathbb R)$   satisfies $u'(\rho)>0,$ $u''(\rho)>0$, then $\omega=i\partial\bar\partial u$  defines a K\"ahler form on $\mathbb C^n\backslash \{0\}$. For $\omega$ to extend to a K\"ahler form on $X$ in the class $a[H]-[E]$,   $u$ must satisfy   boundary asymptotics. Specifically, define  $U_0, U_\infty:[0,\infty)\rightarrow \mathbb R$ via 
$$U_0(r):= u({\rm log}r)-{\rm log}r\qquad {\rm and}\qquad U_\infty(r):= u(-{\rm log}r)+a{\rm log}r.$$
Assume $U_0$ and $U_\infty$    extend by continuity to  smooth functions at $r=0$, with both $U_0'(0)>0$ and  $U_\infty'(0)>0$. This fixes the  asymptotic behavior of $u$
$$\lim_{\rho\rightarrow-\infty}u'(\rho)=1,\qquad\lim_{\rho\rightarrow\infty}u'(\rho)=a,$$
and ensures  $\omega=i\partial\bar\partial u$ extends to a K\"ahler form  on $X$ in the correct class.

Next, given a function $v(\rho)\in C^\infty(\mathbb R)$, the Hessian $i\partial\bar\partial v(\rho)$ defines a $(1,1)$ form $\alpha$ on $\mathbb C^n\backslash \{0\}$. In order for $\alpha$ to extend to $X$ in the class $[\alpha]= p[H]-q[E]$,  we require similar  asymptotics  without the positivity assumptions, as $ \alpha$ need not be a K\"ahler form. Consider the functions    $V_0, V_\infty:[0,\infty)\rightarrow \mathbb R$ defined via
$$V_0(r):= v({\rm log}r)-q{\rm log}r\qquad {\rm and}\qquad V_\infty(r):= v(-{\rm log}r)+p{\rm log}r.$$
Assume that $V_0$ and $V_\infty$  extend by continuity to  smooth functions at $r=0$, which implies $v(\rho)$ satisfies:   
$$\lim_{\rho\rightarrow-\infty}v'(\rho)=q,\qquad\lim_{\rho\rightarrow\infty}v'(\rho)=p.$$
Then $i\partial\bar\partial v $ extends to a smooth (1,1) form on $X$ in the class $[\alpha]$.  

We refer to forms $\omega$ and $\alpha$ constructed in the above manner as having {\it Calabi Symmetry}. Restricting to $\mathbb C^n\backslash \{0\}$, one can check that in this case the eigenvalues of $\omega^{-1}\alpha$ are $\frac {v'}{u'}$ with multiplicity $(n-1)$, and  $\frac{v''}{u''}$ with multiplicity one (for a proof of this see \cite{FL}). Furthermore, because $u''>0$, the first derivative $u'$ is monotone increasing, allowing us to use Legendre transform coordinates and view $u'$ as a real variable, denoted by $x$, which ranges from $1$ to $a$. One can then write $v'$ as a graph $f$ over $x\in(1,a)$, so we have $f(x)=f(u'(\rho))=v'(\rho)$. Taking the derivative of both sides with respect to $\rho$ gives
$$f'(x)u''(\rho)=v''(\rho).$$
We allow the slight abuse of notation where $f'$ denotes the derivative of $f$ with respect to the variable $x$, and $u''$ and $v''$ denote the second derivatives with respect to the variable $\rho$. By the above, the eigenvalues of $\omega^{-1}\alpha$ are 
$$\frac{v'}{u'}=\frac{f}x\,(\text{with multiplicity}\,n-1)\qquad{\rm and}\qquad\frac{v''}{u''}=f'.$$
As $x\rightarrow 1$, we have $\rho\rightarrow -\infty$,  while $x\rightarrow a$ implies $\rho\rightarrow \infty$. Thus the asymptotics of $v(\rho)$ imply 
$$\lim_{x\rightarrow 1^+}f(x)=q,\qquad\lim_{x\rightarrow a^-}f(a)=p,$$
and we extend $f(x)$   to the boundary $[1,a]$ by continuity.

 In this form, the dHYM equation can be written as an ODE
 \begin{align}\label{ODE}
	\textnormal{Im}\left(e^{-i\hat\theta}\left(1+i\frac{f}{x}\right)^{n-1}\left(1+if'\right)\right)=0
\end{align}
subject to the boundary constraints $f(1)=q$, $f(a)=p$. Furthermore the Lagrangian angle given by the eigenvalues of $\omega^{-1}\alpha$ can be expressed as
$$\Theta(x):=(n-1)\arctan\left(\frac fx\right)+\arctan\left(f'\right).$$
Because $\alpha=i\partial\bar\partial v$, in our setting we can write line bundle mean curvature flow as  
$$\dot v=\Theta(x)-\hat\theta.$$
In order to arrive at an equation for $f$ we take the derivative of both sides with respect to $\rho$ and see
$$\frac{d\dot v}{d\rho}=\frac{d\Theta}{dx}\frac{dx}{d\rho}.$$
This now becomes
\begin{align}\label{dHYM flow}
	\dot f=L(f):=u''\left(\frac{f''}{1+f'^2}+(n-1)\frac{xf'-f}{x^2+f^2}\right)= u''\Theta'.
\end{align}

We have now defined second order parabolic equation for $f$, to which a solution can be integrated in $\rho$ to arrive at a solution of the line bundle mean curvature flow \eqref{lbmcf}.  Note that $u''(1)=u''(a)=0$, so the flow fixes the boundary values of $f$. One interesting consequence of taking an extra derivative to define the flow is that it is no longer necessary to take a lift $\hat\theta$ of the average angle. In this way the flow is more analogous to  a how graph evolves by the mean curvature vector rather than how a potential evolves by the Lagrangian angle. 

The flow \eqref{dHYM flow} is defined on graphs over $[1,a]$ with fixed boundary. However, we can generalize it to curves, which is useful in order to construct barriers. Consider the region $D:=\{(x,y)\in\mathbb R^2\,|\, 1\leq x\leq a\}$.   Let $\gamma_t(s):I\subseteq \mathbb R\to D$ be a family of smooth curves, and let $s$ be the arc-length parameter. Let
$$\kappa=\frac{d}{ds}\arctan{\gamma'}$$
denote the usual plane curvature, and let
$$\xi=\frac{d}{ds}\arctan{\gamma}$$
 be an extrinsic quantity. Consider the flow
\begin{align}
\label{mCSF}
	\dot\gamma = u''(\gamma)\left(\kappa+(n-1)\xi\right) \textbf N,
\end{align}
where the normal vector $\textbf N$ is defined by $e^{i\frac\pi2}\gamma'$, and $u''(\gamma)$ is defined to be the function $u''$ applied to the $x$-coordinate of $\gamma$. Notice the relationship between this  flow  and the curve shortening flow $\dot\gamma = \kappa \textbf N.$

In the case where $\gamma(x)=(x,f(x))$ is a graph of a function, we have $ds=\sqrt{1+f'^2} dx$. Simple computations show
  $$\langle\dot\gamma,\textbf N\rangle=\frac{\dot f}{\sqrt{1+f'^2}},  \qquad \kappa=\frac{f''}{(1+f'^2)^{3/2}}, \qquad 	\xi=\frac{1}{\sqrt{1+f'^2}}\frac{xf'-f}{x^2+f^2}.$$
Hence, \eqref{mCSF} reduces to \eqref{dHYM flow} in this case, and thus is the correct generalization to curves.

\section{A finite time singularity}
\label{finite}

Consider a real number $R>1$ (to be determined later), and set $a=6R$. As above, consider a function $u:\mathbb R\rightarrow\mathbb R$ so that $\omega:=i\partial\bar\partial u$ extends from $\mathbb C^n\backslash\{0\}$ to a K\"ahler form on $X$ in the class $a[H]-[E]$. Furthermore, assume that $u''<R$, and that there exists a small constant $k$ so that $u''(x)\geq k(x-1)(a-x)$ on $[1,a]$ (which is possible since by the Calabi symmetry assumptions $u''(x)$ approaches the boundary linearly in $x$). We remark that one should be able to construct a similar singularity example for any K\"ahler form   satisfying  Calabi symmetry, however we include our extra assumptions on $u$ and $R$ for the ease of presentation. 

The idea is as follows. Consider a class $[\alpha]=p[H]-q[E]$ and assume $p\geq a$. Define a representative $\alpha_0$ via the function   $f_0(x)$,  which has a graph such as in Figure 1. We construct a family of shrinking circles, and a traveling family of hyperbolas, which are subsolutions to \eqref{mCSF}. If $f_t$ is the evolution of $f_0$ via the line bundle MCF \eqref{dHYM flow}, and $f_0$ avoids the initial circle and  hyperbola at time $t=0$, then   by the maximum principle $f_t$ must avoid these families for all time. The hyperbolas  push out past the center of the circles before they shrink to a point,  forcing $f_t$ to achieve vertical slope at some finite time. 

\begin{figure}[h!]
  \includegraphics[scale=0.4]{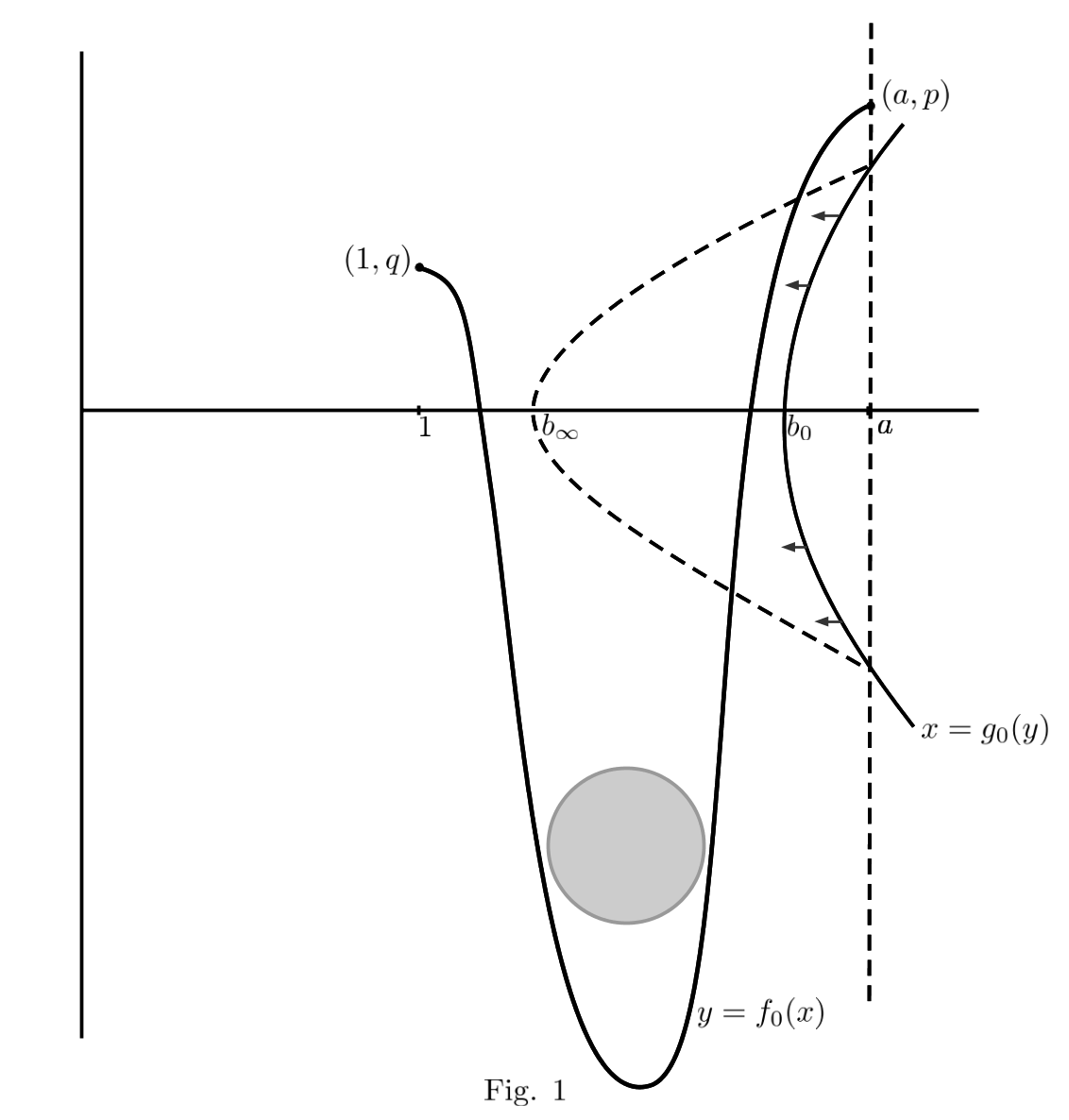}
 
  \caption{The graph of a function $f_0$ which forms a singularity. }
\end{figure}

We first construct our family of hyperbolas. Observe that both $\kappa$ and $\xi$ are invariant under orthogonal transformation. Hence, by interchanging the $x$ and $y$ coordinates, we have the following lemma.

\begin{lemma}\label{rotation}
	Suppose $y=f_t(x)$ satisfies the flow \eqref{dHYM flow}. If the inverse $x=f_t^{-1}(y)=:h_t(y)$ exists, then $h_t(y)$ satisfies 
\begin{align}\label{rdHYM}
	\dot h =u''(h(y))\left(\frac{h''}{1+h'^2}+(n-1)\frac{yh'-h}{y^2+h^2}\right).
\end{align}

\end{lemma}

\begin{lemma}\label{subsolution}
	Suppose $b(t):[0,T)\to \mathbb R$ satisfies the initial value problem:
\begin{align}\label{IVP}
	\dot b=-\frac{kb_{\infty}(b_\infty-1)(a^2-b_0^2)(b-b_\infty)b^3}{a(a^2-b_\infty^2)(2a^2-b_\infty^2)^2}
\end{align}
	where $1<b_\infty <b_0<a$ are constants. Then 
	$$g_t(y):=\sqrt{\frac{a^2-b^2}{a^2-b_\infty^2}y^2+b^2}$$ is a sub-solution to the equation \ref{rdHYM} for $y\in \left[-\sqrt{a^2-b_\infty^2},\sqrt{a^2-b_\infty^2}\right]$.
\end{lemma}
\begin{proof}
	For simplicity, write
$$ m = \frac{a^2-b^2}{a^2-b_\infty^2},\qquad 1-m  = \frac{b^2-b_\infty^2}{a^2-b_\infty ^2}.$$
We also write $g=g_t$ for notational simplicity.  Notice that $b_0\ge b>b_\infty$ from the initial value problem, so $m<1$. We compute 
$$g'=\frac{my}{\sqrt{my^2+b^2}}=\frac{my}g,$$
which in turn gives
$$g''=\frac{m}{g}-\frac{myg'}{g^2}=\frac{m}{g}-\frac{m^2y^2}{g^3}=\frac{mg^2-m^2y^2}{g^3}=\frac{mb^2}{g^3}.$$
Furthermore the two expressions from \eqref{rdHYM} can be written as
$$\frac{g''}{1+g'^2}=\frac{mb^2}{g(g^2+m^2y^2)},$$
and
	$$\frac{yg'-g}{y^2+g^2}= \frac{my^2-g^2}{g(y^2+g^2)}=\frac{-b^2}{{g(g^2+y^2)}}.$$
Thus
\begin{align*}
	L(g)&:=u''(g(y))\left(\frac{g''}{1+g'^2}+(n-1)\frac{yg'-g}{y^2+g^2}\right)\\
	&= u''(g(y))\left(\frac{mb^2}{g(g^2+m^2y^2)}-(n-1)\frac{b^2}{{g(g^2+y^2)}}\right)\\
	&\le u''(g(y))\left(\frac{mb^2}{g(g^2+m^2y^2)}-\frac{b^2}{{g(g^2+y^2)}}\right)\\
	&=u''(g(y))\frac{-(1-m)b^4}{g(g^2+m^2y^2)(g^2+y^2)}\\
	& \le k(g-1)(a-g)\frac{-(1-m)b^4}{g(g^2+m^2y^2)(g^2+y^2)}\le 0,
\end{align*}
where the last inequality follows from our assumption $u''(x)\geq k(x-a)(a-x)$.  

Now, observe that 
$$ (a-g)(a+g)=a^2-my^2-b^2=m\left(\frac{a^2-b^2}m-y^2\right)=m( a^2-b_\infty^2-y^2).$$
The right hand side is non-negative when $y\in\left[-\sqrt{a^2-b_\infty^2},\sqrt{a^2-b_\infty^2}\right]$. As a result
\begin{align*}
L(g) &\leq \frac{-k(g-1)m( a^2-b_\infty^2-y^2)(1-m)b^4}{g(a+g)(g^2+m^2y^2)(g^2+y^2)}\\
& = \frac{-k( a^2-b_\infty^2-y^2)(g-1)(a^2-b^2)(b+b_\infty)(b-b_\infty)b^4}{g(a+g)(a^2-b_\infty^2)^2(g^2+m^2y^2)(g^2+y^2)},
\end{align*}
where we plugged in   the definition of   $m$ and $(1-m)$. Because the above expression is negative, the inequalities $m<1$, $b_\infty \leq g\leq a$, and $b_\infty<b\leq b_0$, allow us to conclude
$$ L(g) \leq \frac{-k( a^2-b_\infty^2-y^2)(b_\infty-1)(a^2-b_0^2)2b_\infty(b-b_\infty)b^4}{g(2a)(a^2-b_\infty^2)^2(2a^2-b_\infty^2)^2}.$$
 
 Next, we turn to the evolution of $g$:
\begin{align*}\dot g =\frac{\dot m y^2+2b\dot b}{2g}=\frac{\frac{-2b\dot b}{a^2-b_\infty^2}y^2+2b\dot b}{2g}&=\frac{b\dot b}{g}\left(1-\frac{y^2}{a^2-b_\infty^2}\right)\\
&=\frac{b\dot b(a^2-b_\infty^2-y^2)}{g(a^2-b_\infty ^2)}\leq 0.
\end{align*}
Putting everything together we arrive at
 \begin{align*}
 	 \dot g-L(g) \geq \frac{b(a^2-b_\infty^2-y^2)}{g(a^2-b_\infty ^2)}\left(\dot b+\frac{k(b_\infty-1)(a^2-b_0^2) b_\infty(b-b_\infty)b^3}{a(a^2-b_\infty^2)^2(2a^2-b_\infty^2)} \right).
 \end{align*}
 The right hand side is zero by the initial value problem. Hence we have demonstrated  $\dot g-L(g)\ge 0$.
 
\end{proof}

We now solve the initial value problem  \eqref{IVP}, and compute the time for which the hyperbola pushes out a specified distance. Set $b_\infty=R$, and $b_0=5R$. Recall $a=6R$, so $1<b_\infty<b_0<a$. Note there exists a constant $M>0$ such that 
$$C_1:=k\frac{b_{\infty}(b_\infty-1)(a^2-b_0^2)}{a(a^2-b_\infty^2)(2a^2-b_\infty^2)^2}\ge \frac{1}{MR^3}.$$
The differential equation \ref{IVP} is separable, yielding
$$-C_1dt =\frac{db}{(b-b_\infty)b^3},$$
which has the solution
$$-C_1t+C_0= \frac{b_\infty^2+2b^2\log(b-b_\infty)+2b_\infty b-2b^2\log(b)}{2b_\infty^3 b^2}$$
 where $C_0$ is given by the initial value $b(0)=b_0=5R$. Plugging in $t=0$ we see directly that
 $$C_0=\frac{11/50+{\rm log}(4/5)}{R^3}.$$
 Let $T$ be the time such that $b(T)=2R$. Then, we have 
\begin{align}\label{time}
	T=\frac1{C_1}\left(\frac{11/50+{\rm log}(4/5)}{R^3}-\frac{5/8-{\rm log}(2)}{R^3}\right)\leq A
\end{align}
for some constant $A$.

\begin{proposition}
\label{hyperbola}
Let $\gamma(t)$ satisfy \eqref{mCSF}. If $\gamma(0)$ does not intersect the hyperbola $g_0(y)$, then $\gamma(t)$ does not intersect $g_t(y)$ for as long as the flow is defined.
\end{proposition}
\begin{proof}
Suppose $p=(x_0, y_0)$ is the first point of intersection of the two curves, occurring at time $t_0$. Since the hyperbola $g_t$ never achieves horizontal slope, we can assume near  $p$ that $\gamma(t)$ is a graph of a function $h_t(y)$ over the ball  $B_\delta(y_0)$ in $y$-axis solving \eqref{rdHYM}. Without loss of generality, for $0\leq t<t_0$ assume that $g_t(y)>h_t(y)$  over $B_\delta(y_0)$. Then over the region $B_\delta(y_0)\times[0,t_0),$ we see $(\frac{d}{dt}-L)(g-h)\geq 0$, yet $g-h>0$ on the parabolic boundary. The result follows from the maximum principle. 
\end{proof}

Next we turn to the family of shrinking circles which act as a barrier. Since $\xi$ is relatively small for a curve far away from the origin,  \eqref{mCSF} behaves similarly to the curve shortening flow in this case. The idea is to consider a family of circles far away from the origin which evolve slightly faster than  curve shortening flow, in order to absorb the small $\xi$ term.
\begin{proposition}\label{circle avoidance}
	For $R=a/6>1$ as above,   assume the graph of   $f_0(x)$ does not intersect the ball $B_{R}(3R,y_0)$. Then, for $y_0$ sufficiently negative, the family of shrinking balls $B_{\sqrt{R^2-4Rt}}(3R,y_0)$ does not intersect the family of graphs of $f_t(x)$ evolving via  \eqref{dHYM flow}, as long as the flow is defined.
\end{proposition}
\begin{proof}
	Locally, we can write $\phi_t(x)=-\sqrt{r(t)^2-(x-3R)^2}+y_0$ as equation   representing the lower boundary of the shrinking balls, where $r(t)=\sqrt{R^2-4Rt}$. Direct computation gives
	\begin{align*}
		u''\frac{\phi''}{1+\phi'^2}-\dot \phi&=\frac{u''-2R}{\sqrt{r^2-(x-3R)^2}}<\frac{-R}{\sqrt{r^2-(x-3R)^2}},
	\end{align*}
since by assumption $u''<R$. Suppose  $t=t_0$ is the first time the graph of $\phi_t$ intersects $f_t$ from above at a point $x_0$. At this point of intersection we have $f'_{t_0}(x_0)=\phi'_{t_0}(x_0)$, $\dot f_{t_0}(x_0)\geq \dot \phi_{t_0}(x_0)$, and we can assume $f_t(x)<\phi_t(x)$ for all $t<t_0$, and so $f''_{t_0}(x_0)\leq \phi''_{t_0}(x_0)$. Then at $t=t_0$, $x=x_0$, we have
	\begin{align*}
		\dot f-\dot \phi&=-\dot \phi+u''\left(\frac{f''}{1+f'^2}+(n-1)\frac{x_0f'-f}{x_0^2+f^2}\right)\\
		&\leq-\dot\phi+ u''\left(\frac{\phi''}{1+\phi'^2}+(n-1)\frac{x_0\phi'-\phi}{x_0^2+\phi^2}\right)\\
		&<  - \frac{R}{\sqrt{r^2-(x_0-3R)^2}}+u''(n-1)\frac{x_0\phi'-\phi}{x_0^2+\phi^2}.
	\end{align*}
To achieve a contradiction we need to show that for $y_0$ sufficiently negative the right hand side above is negative. To control the $\phi'$ term we can compute directly 
$$ - \frac{R}{\sqrt{r^2-(x_0-3R)^2}}+\frac{u''(n-1)x_0\phi' }{x_0^2+\phi^2}=  \frac{-R(x_0^2+\phi^2)+(n-1)u''(x_0-3R)}{(x_0^2+\phi^2)\sqrt{r^2-(x_0-3R)^2}}. $$
Recall that by assumption  $u''<R$. Choose $y_0$ sufficiently negative  to ensure $-(x_0^2+\phi^2)+(n-1)(x_0-3R)\leq -\frac12(x_0^2+\phi^2)$. Then
$$ - \frac{R}{\sqrt{r^2-(x_0-3R)^2}}+\frac{u''(n-1)x_0\phi' }{x_0^2+\phi^2}\leq  \frac{-R}{2\sqrt{r^2-(x_0-3R)^2}}<-\frac12 $$
since $r<R$. We have now demonstrated that 
$$\dot f-\dot \phi<-\frac12-\frac{u''(n-1) \phi}{x_0^2+\phi^2}.$$
The function $\phi$ is negative, so the second term on the right hand side above  is positive. However, we can choose $y_0$ sufficiently negative so that this term is less than $\frac12$, and the result follows.
\end{proof}

We now demonstrate the existence of a singularity using our two subsolutions constructed above.

\begin{figure}[h!]
  \includegraphics[scale=0.4]{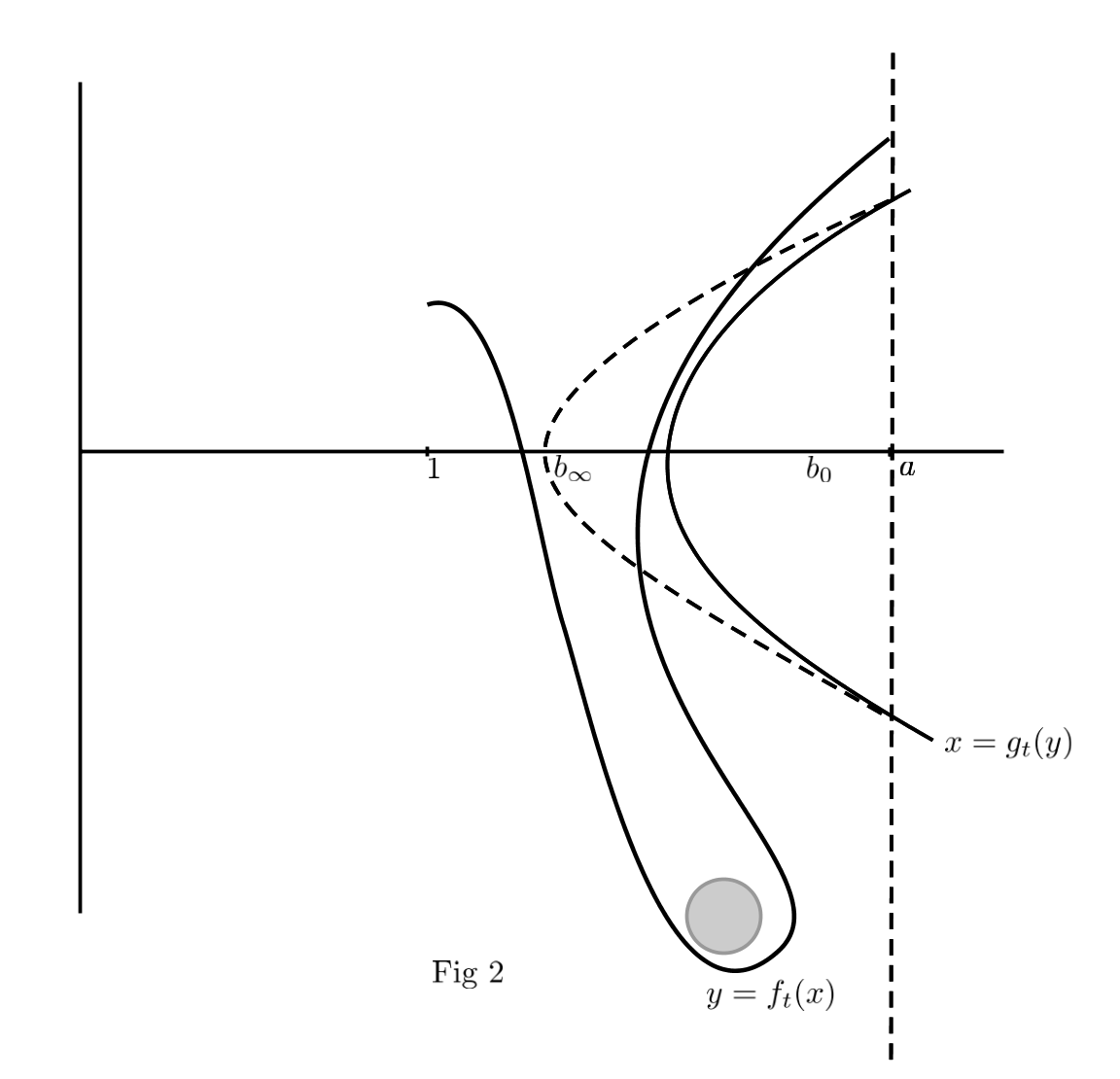}
  \caption{The maximum principle forces $f_t$ to achieve vertical slope. }
\end{figure}

For $R>1$, set $b_\infty=R$, $b_0=5R$, and $a=6R.$  Consider the circle of radius $R$ centered around $(3R, y_0)$, with $y_0$ sufficiently negative so that the hypothesis of Proposition \ref{circle avoidance} is satisfied.  The right side of the circle lies on the line $x=4R$. Note that the vertex of the hyperbola $g_0(y)$ lies on the line $x=b_0=5R$. Furthermore, the hyperbola intersects $x=a$ at $y=\pm\sqrt{35 R^2}$. Since $p>a=6R,$ we see $(a,p)$ lies above the top of the hyperbola $g_0(y)$. Thus, it is possible to choose a function $f_0:[1,a]\rightarrow\mathbb R$ with $f_0(1)=q$, $f_0(a)=p$, such that $f_0$ goes below $B_R(3R, y_0)$, then increases above the hyperbola $g_0(y)$ before arriving at $(a,p)$.

Let $f_t(x)$ be the solution of \eqref{dHYM flow} starting at $f_0(t)$. By Proposition \ref{hyperbola} and Proposition \ref{circle avoidance}, $f_t(x)$ can not intersect $g_t(y)$ nor $B_{\sqrt{R^2-4Rt}}(3R,y_0)$ as long as the flow is defined. Note it takes  time $t=R/4$ for $B_{\sqrt{R^2-4Rt}}(3R,y_0)$ to shrink a point. Also, if $T$ is the time the hyperbola $g_T(y)$ has pushed out to the line $x=2R$, as we have seen by \eqref{time} there exists a constant $A$ such that $T\leq A$. Hence,  choose $R$ large enough to ensure $A<R/4$ and thus $T<R/4$, which implies the hyperbola will push past the center of the shrinking  circles before they completely disappear. This forces $f_t$ to first have a vertical tangency, as illustrated in Figure 2, demonstrating the existence of a finite time singularity and proving Theorem \ref{main}.

\section{Long time existence}
\label{long}
The example above shows that a finite-time singularity for the flow \eqref{dHYM flow} can occur in the interior of the interval $(1,a)$. In particular, one can not always expect $\sup_{(1,a)} |f_t'(x)|$ to stay bounded for finite time. The main goal of this section is to rule out a finite time singularity if one chooses a sufficiently nice initial function $f_0(x)$, specifically where the corresponding $(1,1)$ form $\alpha_0$ has  supercritical phase. This is an important step towards the construction of a singularity at infinite time.

As a first step we show that along the flow, the first   derivative $ |f_t'(x)|$ stays bounded at the boundary points $x=1$ and $x=a$. In fact, our boundary estimate does not need the  supercritical phase assumption.

\begin{proposition}\label{boundary}
	Suppose $f_t(x)$ is defined on $(t,x)\in [0,T)\times[1,a]$. Then, there exists uniform constants $A, B$ so that $$|f_t'(1)|+|f_t'(a)|\leq Ae^{Bt}.$$
\end{proposition}
\begin{proof}
	We will   show  $|f'(1)|<C(T)$, as the other boundry point is treated similarly. Consider $g_t(x)=q+Ae^{B(n-1)t}(x-1)$. Choose $A\gg0$ sufficiently large to ensure both  $Ae^{B(n-1)t}\geq 2\max\{|q|,|q|^{-1}\}$ and $f_0<g_0$  for all $x\in (1,a]$. Choose $B\gg0$ so that $u''<B(x-1)$. We claim that $f_t<g_t$ for all time  $t\in [0,T)$.

	Suppose not, and assume the curves touch for  the first time at $x=x_0>1$ and  $t=t_0$. Then, $f_{t_0}(x_0)=g_{t_0}(x_0)$, $f'_{t_0}(x_0)=g'_{t_0}(x_0)$, $f''_{t_0}(x_0)\le g''_{t_0}(x_0)$, and $\dot f_{t_0}(x_0)\ge \dot g_{t_0}(x_0)$. Thus, when $x=x_0$, $t=t_0$ we have 
\begin{align*}
	\dot f&=u''\left(\frac{f''}{1+f'^2}+(n-1)\frac{x_0f'-f}{x_0^2+f^2}\right)\\
	&\le  B(x_0-1)\left(\frac{g''}{1+g'^2}+(n-1)\frac{x_0g'-g}{x_0^2+g^2}\right)\\
	&=B(x_0-1)(n-1)\frac{Ax_0e^{B(n-1)t_0}-q-Ae^{B(n-1)t_0}(x_0-1)}{x_0^2+(q+Ae^{B(n-1)t_0}(x_0-1))^2}\\
	&< ABe^{B(n-1)t}(x_0-1)(n-1) \frac{1-A^{-1}e^{-B(n-1)t_0}q}{1+q^2},
\end{align*}
	since $x_0^2+(q+Ae^{B(n-1)t_0}(x_0-1))^2>1+q^2$. Furthermore by assumption on $A$ we have	$-A^{-1}e^{-B(n-1)t_0}q\leq  \frac12q^2$, and so $$ \frac{1-A^{-1}e^{-B(n-1)t_0}q}{1+q^2}\leq1.$$
	Hence, 
	$$\dot f< ABe^{B(n-1)t_0}(x_0-1)(n-1)=\dot g,$$
a contradiction. Thus $g_t$ serves as a barrier giving an upper bound for the derivative $f'(1)\leq Ae^{B(n-1)t}$.  The lower bound is treated similarly. 
\end{proof}

We now turn to the case where we do have long time existence, namely when $n\geq 3$ and $\alpha_0$ has supercritical phase, i.e. $\Theta(\alpha_0)>(n-2)\frac\pi 2$. 
\begin{lemma}
The supercritical phase condition is preserved along the flow.
\end{lemma}
\begin{proof}
On a general K\"ahler manifold $(X,\omega)$, set $\omega=ig_{\bar kj}dz^j\wedge d\bar z^k$ and $ \alpha=i\alpha_{\bar kj}dz^j\wedge d\bar z^k$. Consider the metric $\eta_{\bar kj}=g_{\bar kj}+\alpha_{\bar kp}g^{\bar qp}\alpha_{\bar qj}$. By equation (5.4) in \cite{JY},  the angle $\Theta(\alpha_t)$ evolves via the heat equation
\begin{equation}
\label{heat}
\dot\Theta(\alpha_t)=\Delta_{\eta} \Theta(\alpha_t), \end{equation}
 and thus  the result follows from the maximum principle. 
\end{proof}
\begin{lemma}
If $f_0(x)$ satisfies the supercritical phase assumption,   $f_t(x)>0$ for all $t\geq 0$.
\end{lemma}
\begin{proof}
Suppose there exists a time $t_0$ and a point $x_0$ where $f_{t_0}(x_0)\leq 0$. This implies arctan$\left(\frac f{x_0}\right)\leq0$. Yet  $\Theta(x_0):=(n-1)\arctan\left(\frac f{x_0}\right)+\arctan\left(f'\right),$ and so the super critical phase assumption implies
$$\arctan\left(f'\right)>(n-2)\frac\pi 2,$$
which is impossible for $n\geq 3$.
\end{proof}
\begin{lemma}
Under  the supercritical phase assumption there exists a uniform constant $C$ so that $f_t'(x)>-C$ for all $t\geq 0$.
\end{lemma}
\begin{proof}
By the supercritical phase condition
$$ \arctan\left(f_t'\right)> (n-2)\frac\pi2 -(n-1)\arctan\left(\frac {f_t}x\right).$$
Since $x\geq 1$ and $f_t\leq C$ by the maximum principle, there exists an $\epsilon>0$ so that $\arctan\left(\frac {f_t}x\right)<\frac\pi2-\epsilon$. Thus 
$$ \arctan\left(f_t'\right)> -\frac\pi2+(n-1)\epsilon.$$
This gives a lower bound for $f'_t$.
\end{proof}

\begin{proposition}
\label{C1longtime}
	Under  the supercritical phase assumption, a solution $f_t(x)$ to \eqref{dHYM flow} has bounded first derivative for all times $T<\infty$. In particular, there exists uniform constants $A , B $ so that 
	$$\sup_{x\in[1,a]}|f'_t(x)|\leq A(1+t)e^{Bt}.$$
\end{proposition}
\begin{proof}
By the previous lemma we only need an upper bound for $f_t'$. By   Proposition \ref{boundary} we have    $$A ^{-1}e^{-B t}\left(|f_t'(1)|+ |f_t'(a)|\right)\leq 1.$$ As a result if  $\sup_{x\in[1,a]}A^{-1}e^{-Bt}|f'_t(x)|$ is large, this supremum must be achieved at an interior point.  Let $x_0$ be the  interior max. At this point we  have $f'_t(x_0)>0$, $f''_t(x_0)=0$, and  $f'''_t(x_0)\leq 0$. By direct computation  at $x_0$ it holds 
\begin{align*}
	\dot f'&=\frac{d}{dx}\left(u''\left(\frac{f''}{1+f'^2}+(n-1)\frac{xf'-f}{x+f^2}\right)\right)\\
	&\le  \frac{du''}{dx} (n-1)\frac{x_0f'-f}{x_0^2+f^2} + u''\frac {d}{dx}\left(\frac{f''}{1+f'^2}+(n-1)\frac{xf'-f}{x^2+f^2}\right)\\\
	&\leq Cf'+u'' \left(\frac{f'''}{1+f'^2} -(n-1)  \frac{2(x_0f'-f)(x_0+ff') }{(x_0^2+f^2)^2}\right),
\end{align*}
where we repeatedly plugged in that $f''(x_0)=0$. Since $f$ is positive the term $-2x_0f(f')^2$ is negative, and thus
$$\dot f'\leq Cf'+u''2(n-1)  \frac{fx_0+f^2 f'-x_0^2 f' }{(x_0^2+f^2)^2} \leq Cf'+C$$
for some constant $C.$

Now, consider the function $A ^{-1}e^{-B t}f'_t(x)-Ct$. By making $B$ larger, if necessary, we can assume $B \geq C$. At an interior maximum we see
$$\frac{d}{dt}\left(A ^{-1}e^{-B t}f'-Ct\right)\leq 0,$$
from which the result follows. 
\end{proof}

We remark that the above proof fails when the function $f$ is not positive, since then the term $-2x_0f(f')^2$ is positive. Thus the best inequality one can derive in this case is $\dot f'\leq C f'^2$, which is certainly not enough to prevent a finite time singularity, as we have demonstrated.  We are now ready to prove our second main result.

\begin{proof}[Proof of Theorem \ref{main2}] Let  $\alpha_t:=\alpha_0+i\partial\bar\partial \phi_t$,  be the solution to \eqref{lbmcf} starting at $\alpha_0$, and assume the flow is defined for $t\in[0,T)$ for some time $T<\infty$. By proposition \ref{C1longtime}, all the eigenvalues of $\omega^{-1}\alpha_t$ are bounded uniformly by a constant $C_T$. From here the result follows from the argument outlined  in Proposition 5.2 in \cite{JY}. 

The idea is that  once the eigenvalues are bounded, the operator $\Delta_\eta$ is  uniformly elliptic. Given $\Theta(\alpha_t)$ solves the heat equation \eqref{heat}, the  parabolic estimates of Krylov-Safonov  (\cite{Kr} Theorem 11, Section 4.2) imply $\Theta(\alpha_t)$ is in $C^\alpha$ in time which gives $\phi_t$ is uniformly bounded in $C^{1,\alpha}$ in time. Now, the uniform eigenvalue bounds also imply $\phi_t$ has bounded $C^2$ norm. The supercritical phase assumption implies the operator $\Theta(\cdot)$ has convex level sets, which allows us to apply Evans-Krylov theory (see Section 6 of \cite{CJY}). This gives uniform $C^{2,\alpha}$ bounds for $\phi_t$ which can be bootstrapped to higher order estimates. Thus we get smooth convergence $\phi_t\rightarrow\phi_T$ to some limit, which allows us to continue the flow past the time $T$.
\end{proof}

\section{Singular behavior at $t=\infty$}

\label{longsing}

We now construct an example where the line bundle mean curvature flow   develops a singularity at infinite time along a destabilizing subvariety. Recall from Section \ref{symmetry} that if one assumes Calabi-symmetry at an initial time, then \eqref{lbmcf} can be reformulated  as a flow of curves \eqref{mCSF}. As a first step, we construct a family of subsolutions to \eqref{mCSF} in polar coordinates that converges to a   stationary solution $\gamma_\infty$. By \cite{JS}, we know such a solution must lie on a level set of the harmonic polynomial ${\rm Im}(e^{-i\hat\theta}z^n)$. Write $\gamma_\infty(\theta)=(x_\infty(\theta),y_\infty(\theta))=(r_\infty(\theta)\cos\theta,r_\infty(\theta)\sin\theta)$, with $\theta\in[\theta_{\min},\theta_{\max}]$. We also assume
 \begin{align}\label{region}
 	1\le x_\infty(\theta)\le a\quad\textnormal{and}\quad x_\infty(\theta_{\min})=x_\infty(\theta_{\max})=a.
 \end{align}
This leads to the following result.

\begin{proposition}\label{subsolution}
	Under the assumptions
	\begin{equation}
	\label{thing49}
	r_\infty'\ge 0\quad \textit{and}\quad\frac{r_\infty'}{r_\infty}\le 2\tan\theta,
	\end{equation}
	 there exists a subsolution $\gamma_t(\theta)=(r_t(\theta)\cos\theta,r_t(\theta)\sin\theta)$  to \eqref{mCSF} such that $\gamma_t\to \gamma_\infty$ uniformly as $t\to \infty$. 
\end{proposition}

\begin{proof}
We first write down (\ref{mCSF}) in polar coordinates. Note that $\dot \gamma=(\dot r\cos\theta,\dot r\sin\theta)$, with the normal vector to $\gamma$ given by
 $$\textbf N =\frac1{(r'^2+r^2)^{1/2}}(-r'\sin\theta-r\cos\theta,r'\cos\theta-r\sin\theta).$$
 Thus $\langle \dot\gamma,{\textbf N}\rangle=-\frac{\dot r r}{(r'^2+r^2)^{1/2}}$. In this case the extrinsic quantity $\xi$  is simply $ \xi=\frac{d}{ds}\theta=\frac{1}{(r'^2+r^2)^{1/2}}$. The curvature of a plane curve in polar coordinates is given by $\kappa=\frac{2r'^2-rr''+r^2}{(r'^2+r^2)^{\frac32}}$. Hence taking the dot product of \eqref{mCSF} with $\textbf N$ we arrive at
$$\dot rr=-u''\left(\frac{2r'^2-rr''+r^2}{r'^2+r^2}+(n-1)\right).$$
Because $\gamma_\infty$ is stationary, we see (\ref{mCSF}) is equivalent to
\begin{align}\label{curvature}
	\frac{2r_\infty'^2-r_\infty r_\infty''+r_\infty^2}{r_\infty'^2+r_\infty^2}+(n-1)=0.
\end{align}

Now, let $b=b(t):[0,\infty)\to \mathbb R$ be an increasing function to be determined later. We use $b(t)$ to define   $r_t(\theta)$ by
\begin{equation}
\label{thing50}
\frac1{r_t^2(\theta)}=\frac1{1+b}\left(\frac{b}{r_\infty^2(\theta)}+\frac{\cos^2\theta}{a^2}\right).
\end{equation}
For an appropriate choice of $b(t)$, we will show that the family of curves  $\gamma_t(\theta)=(r_t(\theta)\cos\theta,r_t(\theta)\sin\theta)$, which form an interpolation   between $\gamma_0$ and $\gamma_\infty$, gives a subsolution to \eqref{mCSF}. 

\begin{figure}[h!]
\label{figure 3}
  \includegraphics[scale=0.5]{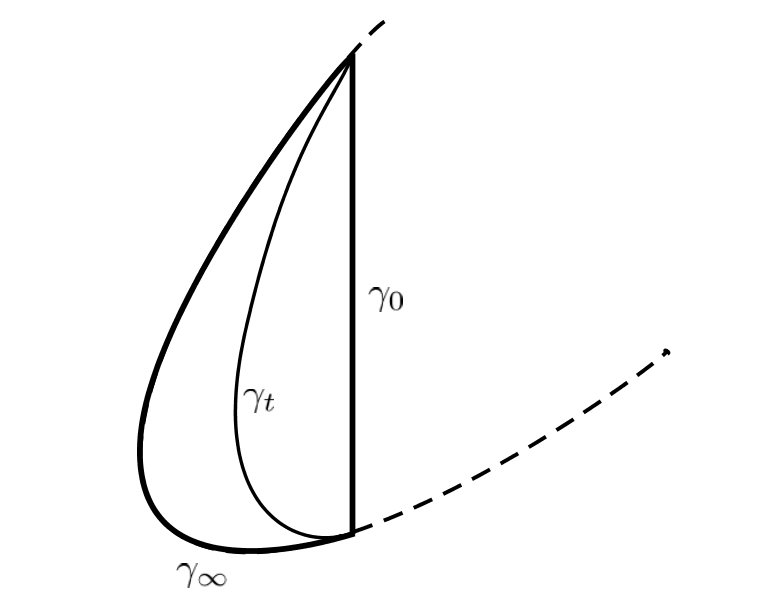}
 
  \caption{$\gamma_t$ being the interpolation between $\gamma_0$ and $\gamma_\infty$. }
\end{figure}

Differentiating \eqref{thing50} with respect to $\theta$, and suppressing dependence on $t$ and $\theta$ from our notation for simplicity, we have
$$\frac{r'}{r^3}=\frac{1}{1+b}\left(\frac{br_\infty'}{r_\infty^3}+\frac{\sin(2\theta)}{2a^2}\right)$$
as well as
$$\frac{r''}{r^3}-\frac{3r'^2}{r^4}=\frac{1}{1+b}\left(\frac{br_\infty''}{r_\infty^3}-\frac{3br_\infty'^2}{r_\infty^4}+\frac{\cos(2\theta)}{a^2}\right).
$$
So,
\begin{align*}
	\frac{2r'^2-rr''+r^2}{r^4}&=-\left(\frac{r''}{r^3}-\frac{3r'^2}{r^4}\right)+\frac{1}{r^2}-\left(\frac{r'}{r^3}\right)^2r^2\\
	&=\frac{1}{1+b}\left(-\frac{br_\infty''}{r_\infty^3}+\frac{3br_\infty'^2}{r_\infty^4}-\frac{\cos(2\theta)}{a^2}+\frac{b}{r_\infty^2}+\frac{\cos^2\theta}{a^2}\right.\\
	&\qquad\qquad\left.-\left(\frac{br_\infty'}{r_\infty^3}+\frac{\sin(2\theta)}{2a^2}\right)^2\left(\frac{b}{r_\infty^2}+\frac{\cos^2\theta}{a^2}\right)^{-1}\right).
\end{align*}
By (\ref{curvature}),
$$-\frac{br_\infty''}{r_\infty^3}+\frac{3br_\infty'^2}{r_\infty^4}+\frac{b}{r_\infty^2}=\frac{-b}{r_\infty^4}\left((n-1)(r_\infty'^2+r_\infty^2)-r_\infty'^2\right).$$
Now, for notational simplicity, set
$$A=\frac{br_\infty'}{r_\infty^3}+\frac{\sin(2\theta)}{2a^2},\;\;\;B=\frac{b}{r_\infty^2}+\frac{\cos^2\theta}{a^2}.$$
Then returning to the above we see
\begin{align*}
	\frac{2r'^2-rr''+r^2}{r'^2+r^2}&=\frac{1}{r^2}\frac{2r'^2-rr''+r^2}{r^4}\left(\left(\frac{r'}{r^3}\right)^2+\left(\frac{1}{r^2}\right)^2\right)^{-1}\\
	&=\frac{B}{A^2+B^2}\left(\frac{-b}{r_\infty^4}\left((n-1)(r_\infty'^2+r_\infty^2)-r_\infty'^2\right)+\frac{\sin^2\theta}{a^2}-\frac{A^2}{B}\right).
\end{align*}
Hence
\begin{align*}
	\frac{2r'^2-rr''+r^2}{r'^2+r^2}+(n-1)&=\frac{1}{A^2+B^2}\left(-(n-1)\frac{Bb}{r_\infty^2}-(n-2)\frac{Bbr_\infty'^2}{r_\infty^4}\right.\\
	&\quad\quad \left.+\frac{B\sin^2\theta}{a^2}+(n-2)A^2+(n-1)B^2\right).
\end{align*}

We now compute
$$-(n-1)\frac{Bb}{r_\infty^2}+(n-1)B^2=(n-1)B\left(B-\frac{b}{r_\infty^2}\right)=(n-1)B\frac{\cos^2\theta}{a^2},$$
and
$$A^2-\frac{Bbr_\infty'^2}{r_\infty^4}=\frac{br_\infty'\sin(2\theta)}{ a^2r_\infty^3}+\frac{\sin^2(2\theta)}{4a^4}-\frac{br_\infty'^2\cos^2\theta}{ a^2r_\infty^4}.$$
Combining these, we have
\begin{align*}
	\frac{2r'^2-rr''+r^2}{r'^2+r^2}+(n-1)&= \frac{(n-1)B\cos^2\theta+B\sin^2\theta}{a^2(A^2+B^2)}+\frac{(n-2)\sin^2(2\theta)}{4a^4(A^2+B^2)}\\
	&+\frac{n-2}{A^2+B^2}\left(\frac{br_\infty'\sin(2\theta)}{ a^2r_\infty^3}-\frac{br_\infty'^2\cos^2\theta}{ a^2r_\infty^4}\right).
\end{align*}
By assumption, $$r_\infty'\ge 0\quad \textnormal{and}\quad\frac{r_\infty'}{r_\infty}\le 2\tan \theta,$$
which implies
$$\frac{br_\infty'\sin(2\theta)}{ a^2r_\infty^3}-\frac{br_\infty'^2\cos^2\theta}{ a^2r_\infty^4}\ge 0.$$
Additionally, $r_\infty,$ $\sin\theta$, and $\cos\theta$, are all bounded above and below away from zero. This implies there exists a constant $C_1$ so that, for large $b$, 
$$\frac{2r'^2-rr''+r^2}{r'^2+r^2}+(n-1)\ge \frac{C_1}b.$$

Returning to \eqref{thing50}, we take the derivative of both sides in $t$
$$-\frac{2\dot r}{r^3}=-\frac{\dot b}{(1+b)^2}\left(\frac b{r_\infty^2}+\frac{\cos^2\theta}{a^2}-\frac{1+b}{r_\infty^2}\right)=-\frac{\dot b}{(1+b)^2}\left( \frac{\cos^2\theta}{a^2}-\frac{1}{r_\infty^2}\right).$$
Multiplying by $-r^4$ and plugging in the square of \eqref{thing50} for $r^4$ gives
 \begin{align*}
	2r\dot r&=\left(\frac{\cos^2\theta}{a^2}-\frac{1}{r_\infty^2}\right)\left(\frac{b}{r_\infty^2}+\frac{\cos^2\theta}{a^2}\right)^{-2}\dot b\\
	&= (r_\infty x-ra)\left(\frac{r_\infty x+ra}{a^2r^2r_\infty^2}\right)\left(\frac{b}{r_\infty^2}+\frac{\cos^2\theta}{a^2}\right)^{-2}\dot b\\
	&\ge (r_\infty x-ra)\frac{C_2}{b^2} \dot b
\end{align*}
for some $C_2>0$ whenever $b$ is large. Note that the polar curves $r(\theta)$ intersect  the line $x=a$ to the zeroth order, which implies there exists a constant $C_3>0$ for which \begin{align}\label{C3}
	0\ge \inf_{x\in [a-\epsilon,a]} \left(u''^{-1}(r_\infty x-ra)\right)+\inf_{x\in[1,a]}(r_\infty x-ra)\ge -C_3.
\end{align}

Next, we use the same assumption on the background K\"ahler form as Section \ref{finite}, namely, for $x\in [1,a-\epsilon]$ we assume $u''(x)\ge k(x-1).$ This implies
\begin{align*}
	u''&\ge k(r\cos\theta -1)\\
	&=k\left(\sqrt{(1+b)\left(\frac{b}{r_\infty^2}+\frac{\cos^2\theta}{a}\right)^{-1}}\cos\theta-1\right)\\
	&=k\left(\sqrt{(1+b)\left(\frac{b}{(r_\infty\cos\theta)^2}+\frac{1}{a}\right)^{-1}}-1\right)\\
	&\ge k\left(\sqrt{(1+b)\left(b+\frac{1}{a}\right)^{-1}}-1\right).
\end{align*}
For simplicity, write the right hand side above as $C(b)$, which is a smooth positive function approaching $0$ as $b\to \infty$. Combining with (\ref{C3}) we arrive at, 
\begin{align*}
	\frac2{u''}r\dot r\ge -\frac{C_2C_3}{b^2}\left(1+\frac 1{C(b)}\right)\dot b.
\end{align*}
If $b$ solves the initial value problem
$$\dot b=2\left(1+\frac{1}{C(b)}\right)^{-1}\frac{C_1}{C_2C_3}b;\qquad b_0\gg 0,$$
then $r_t(\theta)$ defines a subsolution:
\begin{align*}
	&\frac{1}{u''}r\dot r+\left(\frac{2r'^2-rr''+r^2}{r'^2+r^2}+(n-1)\right)\ge 0.
	\end{align*}
Notice that we require $ b_0\gg 0$. Thus the subsolution does not start at $\gamma_0$ (given by the vertical line in Figure 3), but rather a curve starting closer to $\gamma_\infty$ in the interpolation. It then sweeps out to $\gamma_\infty$ as $t\rightarrow\infty$. 
\end{proof}

We now show that the assumptions on $r_\infty$ in Proposition \ref{subsolution} can be satisfied with an explicit example. As we have stated above, in  \cite{JS} it was demonstrated that under the Calabi-Symmetry assumption, solutions to the dHYM equation correspond to functions $f:[1,a]\rightarrow \mathbb R$, satisfying the boundary conditions $f(1)=q$, $f(a)=p$, so that the graph $(x, f(x))$ lies on a level curve of ${\rm Im}(e^{-i\hat\theta}z^n)$. Furthermore, the proof of Theorem 1 from \cite{JS} uses that if the level curve through $(1,q)$ has vertical slope, then the  class $[\alpha]$ is semi-stable with respect to the stability condition \eqref{stability1}, with the exceptional divisor $E$ being the destabilizing subvariety. Thus in this case any graph $f_\infty(x)$ lying on the level curve is singular with unbounded derivative at $(1,q)$, and by construction the corresponding representative of $[\alpha]$ will be singular precisely along $E$. It is this singular graph that will be the limiting curve to the line bundle mean curvature flow. 
\begin{lemma}
	There exists a K\"ahler class $[\omega]$ and a semi-stable class $[\alpha]$ with a stationary solution $\gamma_\infty$  which satisfies $\gamma_\infty(\theta_0)=(1,q)$, $\gamma_\infty(\theta_{\max})=(a,p)$, and where the corresponding polar function $r_\infty$ satisfies the assumptions of Proposition \ref{subsolution}.	\end{lemma}
\begin{proof}

Choose  $\gamma_\infty$ lying on a level curve of ${\rm Im}(e^{-i\hat\theta}z^n)$ so that  $\gamma_\infty(\theta_0)=(1,q)$ and  $\gamma_\infty'(\theta_0)$ is vertical, for some $\theta_0$. This guarantees we are working with a semi-stable class. The corresponding polar function $r_\infty(\theta)$ will now satisfy  (\ref{curvature}). Define $\beta$ by
	\begin{align*}
		\tan \beta:=\frac{y'(\theta)}{x'(\theta)}&=\frac{r_\infty'\sin\theta+r_\infty\cos\theta}{r_\infty'\cos\theta-r_\infty\sin\theta}=\frac{r_\infty'r_\infty^{-1}\tan\theta+1}{r_\infty'r_\infty^{-1}-\tan\theta}.
	\end{align*}
As a result
	$$\frac{r_\infty'}{r_\infty}=\cot(\beta-\theta)=\tan(\pi/2-\beta+\theta).$$
	Now, choose $q\gg 0$. Because $\gamma_\infty'(\theta_0)$ is vertical, we know $\beta(\theta_0)=\pi/2$. In particular, at this point
	$$r_\infty'(\theta_0)>0\quad\textnormal{and}\quad \frac{r_\infty'(\theta_0)}{r_\infty(\theta_0)}=\tan(\theta_0)<2\tan(\theta_0).$$
	Thus, there exists a neighborhood of $\theta_0$ where \eqref{thing49} holds. 
	
	We now check (\ref{region}). At $\theta=\theta_0$, 
	\begin{align*}
		x_\infty'&=r_\infty\cos\theta\left(\frac{r_\infty'}{r_\infty}-\tan\theta\right)=0\\
		x_\infty''&=\frac{\cos\theta}{r_\infty}\left(-2r_\infty r_\infty'\tan\theta+r_\infty r_\infty''-r_\infty^2\right)\\
		&=\frac{\cos\theta}{r_\infty}\left(-2r_\infty'^2+r_\infty r_\infty''-r_\infty^2\right)>0,
	\end{align*}
where last inequality follows from (\ref{curvature}). Hence, $x_\infty$ achieves local minimum at $\theta=\theta_0$. We choose $a$ slightly greater than $1$ such that $x_\infty(\theta_{\min})=x_\infty(\theta_{\max})=a$. This demonstrates the assumptions of Proposition \ref{subsolution}.
\end{proof}

We are now ready to complete the proof of Theorem \ref{main3}. Consider the  classes $[\omega]$ and  $[\alpha]$ discussed in the above lemma. Let $f_\infty(x)$ denote the graphical portion of $\gamma_\infty$ that connects $(1,q)$ to $(a,p)$. Since the assumptions of Proposition \ref{subsolution} are satisfied, there exists a subsolution $\gamma_t$  pushing out towards $\gamma_\infty$. In the proof of Proposition \ref{subsolution} we saw the subsolution condition is not satisfied unless $b$ is sufficiently large, and so the subsolution starts at some time $t_0$, with $\gamma_{t_0}$  already pushed out   towards $\gamma_\infty$.

Consider a function $f_{t_0}$ satisfying $f_{t_0}(1)=q$, and $f_{t_0}(a)=p$, which lies above the curve $\gamma_{t_0}$, but below $\gamma_\infty$, as in Figure 4. This function defines an initial representative $\alpha_0\in[\alpha]$, and its angle   is given by
$$\Theta(\alpha_0)=(n-1){\rm arctan}\left(\frac{f_{t_0}}{x}\right)+{\rm arctan}(f_{t_0}').$$
The supercritical phase assumption in Theorem \ref{main2} is satisfied if we choose $q$ large enough so that ${\rm arctan}\left(\frac{f_{t_0}}{x}\right)$ is sufficiently  close to  $\pi/2$. Thus if we consider a solution $\alpha_t$ to the line bundle mean curvature flow starting at $\alpha_0$, the flow exists for all time. Let $f_t$ be graph corresponding to $\alpha_t$. By the maximum principle, $f_t$ must stay below $f_\infty$ and above $\gamma_t$ for all time.

\begin{figure}[h!]
  \includegraphics[scale=0.5]{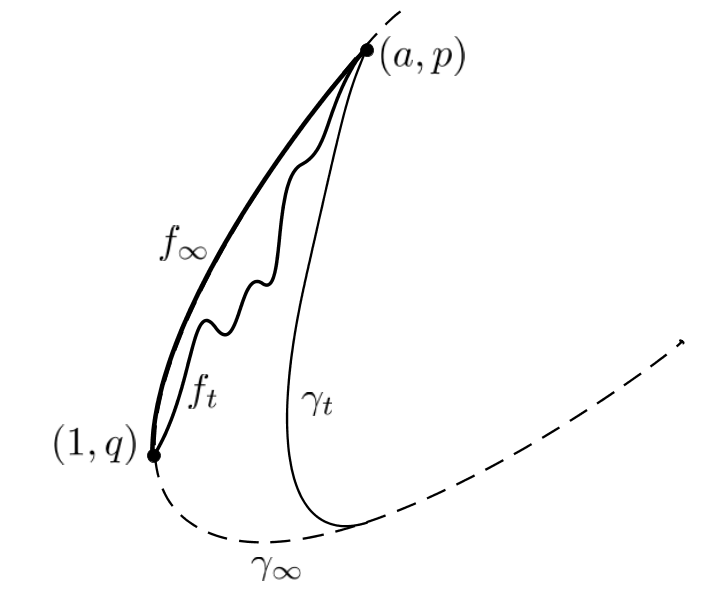}
 
  \caption{A singularity at $(1,q)$ at $t=\infty$.}
\end{figure}

 Because the subsolution $\gamma_t$ sweeps out to $\gamma_\infty$ as $t\rightarrow\infty$, the solution to the flow $f_t$ must converge to $f_\infty$ in $C^0$. In particular, it can not develop an infinite time singularity at any point other than $(1,q)$, where it will achieve vertical tangency. By construction, this point corresponds to the exceptional divisor $E$, which is precisely the destabilizing subvariety. Thus, the corresponding  forms  $\alpha_t$ along the line bundle mean curvature flow will blow up along $E$ at infinite time.


\begin{thebibliography}{4}


\bibitem{Ca} E. Calabi, {\it Extremal K\"ahler metrics}, Seminar on Differential Geometry, Vol. { 102} of Ann. Math. Studies, Princeton Univ. Press, Princeton, N.J. (1982), 259-290. 

\bibitem{C1} G. Chen, {\it The J-equation and the supercritical deformed Hermitian-Yang-Mills equation.}
Invent. Math. 225 (2021), no. 2, 529-602.


\bibitem{CJY} T.C. Collins, A. Jacob, and S.-T. Yau, {\it (1,1) forms with specified Lagrangian phase.} Camb. J. Math. 8 (2020), no. 2, 407-452.



\bibitem{CLSY}  T.C. Collins, J. Lo,   Y. Shi, and S.-T. Yau, {\it }Bridgeland stable line bundles of fiber degree 1 on Weierstrass elliptic K3 surface. arXiv:2306.05620 .


\bibitem{CS} T.C. Collins and Y. Shi, {\it Stability and the deformed Hermitian-Yang-Mills equation}. Surveys in Differential Geometry. 24 (2019), 1-38.

\bibitem{CXY} T.C. Collins, D. Xie, and S.-T. Yau, {\it The deformed Hermitian-Yang-Mills equation in geometry and physics}, 	Geometry and Physics, Volume I: A Festschrift in honour of Nigel Hitchin, Oxford University Press, December, 2018.

\bibitem{CY} T.C. Collins and S.-T. Yau, {\it Moment maps, nonlinear PDE, and stability in mirror symmetry}, 	Ann. PDE 7 (2021), no. 1, Paper No. 11, 73 pp.

\bibitem{CLT}  J. Chu, M.-C. Lee, and R. Takahashi, {\it  A Nakai-Moishezon type criterion for supercritical deformed Hermitian-Yang-Mills equation} 	arXiv:2105.10725.


\bibitem{FL} H. Fang and M. Lai, {\em Convergence of general inverse $\sigma_k$-flow on K\"ahler manifolds with Calabi ansatz},
Trans. Amer. Math. Soc. 365 (2013), no. 12, 6543-6567. 

\bibitem{FYZ} J. Fu, S.-T. Yau, and D. Zhang, {\it A deformed Hermitian-Yang-Mills flow}. 	arXiv:2105.13576

\bibitem{HJ2}X. Han and  X Jin, {\it Chern number inequalities of deformed Hermitian-Yang-Mills metrics on four dimensional K\"ahler manifolds}. 	arXiv:2008.06862.


\bibitem{HJ} X. Han and  X Jin, {\it Stability of line bundle mean curvature flow}, arXiv: 2001.07406.
 
 \bibitem{J} A. Jacob, {\it The Deformed Hermitian-Yang-Mills Equation and Level Sets of Harmonic Polynomials}. 	arXiv:2204.01875.

\bibitem{JS} A. Jacob and N. Sheu, {\it The deformed Hermitian-Yang-Mills equation on the blowup of $\mathbb P^n$}. 


\bibitem{JY} A. Jacob and S.-T. Yau, {\em A special Lagrangian type equation for holomorphic line bundles}, Math. Ann. 369 (2017). no 1-2, 869-898.




\bibitem{Kr} N.V. Krylov, {\it Nonlinear elliptic and parabolic equations of the second order}. Mathematics and its Applications (Soviet Series), 7. D. Reidel Publishing Co., Dordrecht, 1987. xiv+462 pp. ISBN: 90-277-2289-7. 


\bibitem{LS} M. Lejmi, and G. Sz\'ekelyhidi, {\em The $J$-flow and stability}, Advances in Math. 274 (2015), 404-431.


\bibitem{LYZ} C. Leung, S.-T. Yau, and E. Zaslow, {\it
From special Lagrangian to Hermitian-Yang-Mills via Fourier-Mukai transform,} Winter School on Mirror Symmetry, Vector Bundles and Lagrangian Submanifolds (Cambridge, MA, 1999), 209-225, AMS/IP Stud. Adv. Math., 23, Amer. Math. Soc., Providence, RI, 2001. 


\bibitem{MMMS} M. Marino, R. Minasian, G. Moore, and A. Strominger, {\it Nonlinear Instantons from Supersymmetric p-Branes,} hep-th/9911206. 

\bibitem{Ping} P. Pingali, {\it The deformed Hermitian Yang-Mills equation on three-folds}. Anal. PDE 15 (2022), no. 4,
921-935.



\bibitem{S1} K. Smoczyk, {\it Angle theorems for the Lagrangian mean curvature flow}, Math. Z. 240 (2002), 849-883.

\bibitem{S2} K. Smoczyk, {\it Longtime existence of the Lagrangian mean curvature flow,} Calc. Var. Partial Differential Equations 20 (2004), 25-46. 

\bibitem{SmW2} K. Smoczyk and M.-T. Wang, {\it Mean curvature flows of Lagrangians submanifolds 
with convex potentials,} J. Differential Geom. 62 (2002), 243-257.

\bibitem{Song} J. Song, {\it Ricci flow and birational surgery}, arXiv:1304.2607.

\bibitem{SW2} J. Song and B. Weinkove, {\it Contracting exceptional divisors by the K\"ahler-Ricci flow},   Duke
Math. J. 162 (2013), no. 2, 367-415.

\bibitem{SW3} J. Song and B. Weinkove, {\it Contracting exceptional divisors by the K\"ahler-Ricci flow II}, Proc.
Lond. Math. Soc. (3) 108 (2014), no. 6, 1529-1561.



\bibitem{SW4} J. Song and B. Weinkove, {\it The K\"ahler-Ricci flow on Hirzebruch surfaces}, J. Reine Angew.
Math. 659 (2011), 141-168.


\bibitem{SY} J. Song and Y. Yuan, {\it Metric flips with Calabi ansatz},   Geom. Func. Anal. 22 (2012), no.
1, 240-265.




\bibitem{Tak2} R. Takahashi, {\it Collapsing of the line bundle mean curvature flow on K\"ahler surfaces}. Calc. Var. Partial
Differential Equations 60 (2021), no. 1, Paper No. 27, 18 pp.


\bibitem{Tak1} R. Takahashi, {\it Tan-concavity property for Lagrangian phase operators and applications to the tangent
Lagrangian phase flow}. Internat. J. Math. 31 (2020), no. 14, 2050116, 26 pp.


\bibitem{TTW} C.-J. Tsai, M.-P. Tsui, and M.-T. Wang, {\it Mean curvature flows of two-convex Lagrangians}. arXiv:2302.02512.
 
\bibitem{W} M.-T. Wang, {\it Long-time existence and convergence of graphic mean curvature flow in arbitrary codimension}, Invent. math. 148 (2002) 3, 525-543.    



\bibitem{Z} J. Zhang, {\it  A note on the supercritical deformed Hermitian-Yang-Mills equation}.	arXiv:2302.06592.
\end{thebibliography}
\end{document}